\documentclass[a4paper]{article}
\usepackage[english]{babel}
\usepackage[latin1]{inputenc}
\usepackage{amsmath,amsfonts,amssymb,amsthm,mathrsfs,paralist,authblk,cite,dsfont,aliascnt}
\usepackage[bookmarksnumbered]{hyperref}

\DeclareMathOperator{\Real}{Re}
\DeclareMathOperator{\Imag}{Im}
\newcommand{\eps}{\varepsilon}
\newcommand{\NN}{\ensuremath{\mathbb{N}}}
\newcommand{\RR}{\ensuremath{\mathbb{R}}}
\newcommand{\CC}{\ensuremath{\mathbb{C}}}
\newcommand{\FF}{\ensuremath{\mathbb{F}}}
\newcommand{\GG}{\ensuremath{\mathbb{G}}}
\newcommand{\LL}{\ensuremath{\mathbb{L}}}
\newcommand{\AAA}{{\mathcal A}}
\newcommand{\CCC}{{\mathcal C}}
\newcommand{\EEE}{{\mathcal E}}
\newcommand{\FFF}{{\mathcal F}}
\newcommand{\JJJ}{{\mathcal J}}
\newcommand{\PPP}{{\mathcal P}}
\newcommand{\HHHH}{{\mathscr H}}

\numberwithin{equation}{section}
\newtheorem{defn}{Definition}[section]   
\newaliascnt{lem}{defn}
\newtheorem{lem}[lem]{Lemma}
\aliascntresetthe{lem}  

\newaliascnt{prop}{defn}
\newtheorem{prop}[prop]{Proposition}
\aliascntresetthe{prop}  

\newaliascnt{thm}{defn}
\newtheorem{thm}[thm]{Theorem}
\aliascntresetthe{thm}  

\newaliascnt{cor}{defn}
\newtheorem{cor}[cor]{Corollary}
\aliascntresetthe{cor}  

\theoremstyle{definition}
\newaliascnt{expl}{defn}

\aliascntresetthe{expl}  

\newaliascnt{expls}{defn}

\aliascntresetthe{expls}  

\newaliascnt{rem}{defn}

\aliascntresetthe{rem}  

\newaliascnt{rems}{defn}

\aliascntresetthe{rems}

\begin{document}

\title{
Predictable projections of conformal stochastic integrals: an application to Hermite series and to Widder's representation
}

\author{
Matteo Casserini and Freddy Delbaen\thanks{e-mail: \texttt{matteo.casserini@math.ethz.ch}, \texttt{delbaen@math.ethz.ch}. The authors would like to thank Marc Yor (Universit\'e Pierre et Marie Curie, Paris) and Larbi Alili (Warwick University) for helpful discussions and comments.}
}

\affil{
\small{Department of Mathematics, ETH Z\"urich, Switzerland}
}

\date{December 23, 2011}
\maketitle

\begin{abstract}
In this article, we study predictable projections of stochastic integrals with respect to the conformal Brownian motion, extending the connection between powers of the conformal Brownian motion and the corresponding Hermite polynomials. As a consequence of this result, we then investigate the relation between analytic functions and $L^p$-convergent series of Hermite polynomials. Finally, our results are applied to Widder's representation for a class of Brownian martingales, retrieving a characterization for the moments of Widder's measure.


\end{abstract}

\section{Introduction}

The purpose of this article is to introduce some complexification techniques for stochastic processes, that allow to consider real-valued processes as appropriate projections of corresponding complex-valued, conformal stochastic processes. As an application of our complexification techniques, we derive a characterization of Widder's integral representation for Brownian martingales, which is obtained by adapting to the probabilistic setting a classical result for the heat equation \cite{Wid}.

We start by studying predictable projections in a conformal Brownian setting. Conformal martingales have been  introduced by Getoor and Sharpe \cite{GetSha} to prove the duality between the Hardy space $H^1$ and the space $BMO$ in the martingale setting: conformal martingales later played an important role in the probabilistic study of analytic functions as well as in the derivation of the conformal invariance of Brownian motion (see for instance the survey article \cite{Dav}).

While stochastic integration with respect to conformal martingales is particularly interesting because of the properties of the complex plane, to our knowledge there has not been any attempt to introduce a notion of projection of such integrals on the real line. As a first step in this direction, we consider the predictable projection on the real component of a conformal Brownian motion. It turns out that such a projection behaves well under integration, and in particular powers of the conformal Brownian motion project onto the corresponding Hermite polynomials. Such a remarkable property stresses once more the importance of Hermite polynomials in stochastic analysis (which is due especially to their close relation with iterated stochastic integrals and the Wiener chaos decomposition, see for instance Nualart \cite{Nua}), and it motivates the subsequent study of series of Hermite polynomials, allowing us to obtain, in a stochastic setting, interesting connections to analytic functions.

In the second part, the techniques derived previously are applied to a wide class of Brownian martingales, obtaining a further characterization of Widder's representation. We recall that, by the results of Widder \cite{Wid}, any positive solution of the heat equation can be rewritten in terms of a Laplace-Stieltjes integral with respect to some measure $\mu$, which however remains undetermined. We will show that the quadratic exponential moments of $\mu$ can be characterized by applying our results on series of Hermite polynomials and related power series of conformal Brownian motion. Moreover, we obtain a relation between Widder's representation and a particular class of analytic functions.

The article is organized as follows. In Section $2$, we recall the notion of predictable projections of stochastic processes and show how stochastic integrals with respect to the conformal Brownian motion are projected on the real line. Then, in Section $3$ we derive $L^p$-convergence properties for series of Hermite polynomials from well known $L^p$-estimates on the Wiener chaos. Section $4$ is dedicated to the presentation in a purely probabilistic setting of Widder's representation result as well as its extension to $L^1$-bounded martingales. Finally, we derive in Section $5$ the characterization of the moments of Widder's measure $\mu$, as well as the aforementioned connection to analytic functions.

\section{Predictable projections of stochastic integrals}

We begin by introducing some notation. Let $(\Omega, \FFF, P)$ be a complete probability space, and assume that $X$, $Y$ are two independent, $d$-dimensional Brownian motions on $(\Omega, \FFF, P)$. We denote by $Z$ the conformal $d$-dimensional Brownian motion given by $Z=X + i \, Y$. Furthermore, let $\FF=(\FFF_t)_{t \ge 0}$ be the augmented filtration generated by $Z$, and let $\FF^X=(\FFF_t^X)_{t \ge 0}$, $\FF^Y=(\FFF_t^Y)_{t \ge 0}$ denote the filtrations generated by $X$, respectively $Y$, and augmented by the $P$-nullsets from $\FF$. Unless otherwise stated, we will always define stochastic integrals with respect to the filtered probability space $(\Omega, \FFF, \FF, P)$. We denote by $ b \LL$ the space of all adapted processes with bounded càglàd paths, and by $b \PPP$ the space of all bounded predictable processes. Moreover, we define as usual
\begin{align*}
\HHHH^2(Z):=\bigg\{ H: \Omega \times [0,T] \to \CC^d \ \bigg| \ &H \text{ predictable with respect to } \PPP^{\FF}, \text{ and } \\
&\|H\|_{\HHHH^2(Z)}:=E\left[\int_0^\infty |H_t|^2 d t\right] < \infty \bigg\},
\end{align*}
where $|\cdot|$ is the Euclidean norm, and similarly for $\HHHH^2(X)$. Finally, we denote by $\Pi^X$ the orthogonal projection from $\HHHH^2(Z)$ onto the space $\HHHH^2(X)$. We shortly recall the definition of the predictable projection of a measurable process:

\begin{defn}
Let $\GG$ be a filtration on $(\Omega, \FFF, P)$, and let $\PPP^{\GG}$ denote the predictable $\sigma$-field with respect to $\GG$. Then, for a measurable process $L$ such that $L$ is positive or bounded, there exists a unique process $\widetilde{L}$, measurable with respect to $\PPP^{\GG}$ such that, for every predictable $\GG$-stopping time $T$,
\[
E[L_T|{\mathcal G}_{T-}]=\widetilde{L}_T \; P\text{-a.s. on } \{T < \infty\}.
\] 
$\widetilde{L}$ is then called the predictable projection of $L$ on $\GG$.

\end{defn}

An easy calculation shows that, for stochastic processes $H$ that are optional for $\FF^Z$ and such that $E[\int_0^T |H_u|^2\,du]<\infty$, the predictable projection on $\FF^X$ is a version of the projection $\Pi^X(H)$. There is a slight distinction that shows that the predictable projection is a finer object than the orthogonal projection.  Indeed, the predictable projection is a process, defined up to evanescent sets, whereas the orthogonal projection is a class of random variables defined up to sets of $dP\times dt$ measure zero.

In the following, we will concentrate our attention on the predictable projection on $\FF^X$: because of the properties of the Brownian motion $X$, this actually coincides with the optional projection on $\FF^X$. To simplify our notation, the predictable projection of $L$ on $\FF^X$ will be denoted by $L^{\PPP^X}$.

We can now prove our first main result: the predictable projection on $\FF^X$ maps stochastic integrals with respect to $Z$ onto stochastic integrals with respect to $X$. Moreover, we also obtain a relation between the integrand processes.

\begin{thm}\label{predproj}
Let $H$ be a process such that $H \in \HHHH^2(Z)$. Then, the predictable projections $\left(\int H d Z\right)^{\PPP^X}$ exists, and
\[
\left(\int H d Z\right)^{\PPP^X}_t=\int_0^t \Pi^X(H) \, d X \; P\text{-a.s. for all } t \ge 0.
\]
\end{thm}

\begin{proof}
First of all, we observe that both stochastic integrals can be realized on the filtered probability space $(\Omega, \FFF, \FF, P)$, as $X$ and $Z$ are both continuous martingales on it. Moreover, we notice that the existence of $\left(\int H d Z\right)^{\PPP^X}$ is a consequence of classical results on filtration shrinkage, which can be found for instance in \cite{ProBook}.

We first assume that $H \in b \LL$. Fix $t \ge 0$, and consider a sequence $(\pi^n)_{n \in \NN}$ of partitions of $[0,t]$ such that $ |\pi^n| \to 0$. Because of classical convergence results in stochastic analysis (see \cite{ProBook}), we have that
\[
\int_0^t H_s d Z_s=\lim_{n \to \infty} \sum_{t_i \in \pi^n} H_{t_i} 
(Z_{t_{i+1}} -Z_{t_i}) \quad \text{in } \HHHH^2,
\]
and hence there is a subsequence $(\pi^{n_k})_{k \in \NN}$ so that $\sum_{t_i \in \pi^{n_k}} H_{t_i} (Z_{t_{i+1}} -Z_{t_i})$ converges to $\int_0^t H_s d Z_s$ $P$-a.s. as $k \to \infty$. Since $H$ is bounded, the bounded convergence theorem gives that
\begin{multline*}
\left(\int H d Z\right)^{\PPP^X}_t=E\bigg[\int_0^t H_s d Z_s\bigg|\FFF^X_t\bigg]
=\lim_{k \to \infty} \sum_{t_i \in \pi^{n_k}} E[H_{t_i} (Z_{t_{i+1}} -Z_{t_i})|\FFF^X_t] \\
=\lim_{k \to \infty} \sum_{t_i \in \pi^{n_k}} \left(E[H_{t_i} (X_{t_{i+1}} -X_{t_i})|\FFF^X_t] + i \,  E[H_{t_i} (Y_{t_{i+1}} -Y_{t_i})|\FFF^X_t] \right).
\end{multline*}
We compute the first term. Consider for $t \ge 0$ the class 
\[\CCC_t:=\{C \in \FFF \ | \ \exists \, A \in \FFF^X_t, B \in \FFF^Y_t \text { such that } C=A \cap B\},
\] 
which is stable under intersection. Because of the independence of $X$ and $Y$, we can compute that, for all $F \in L^1(\FFF)$ and $C=A \cap B \in \CCC_{t_i}$,
\begin{align*}
E\big[E[F |\FFF^X_t] \mathds{1}_C\big]&=E\big[E[F |\FFF^X_t] \mathds{1}_A \mathds{1}_B\big] \\
&=E\big[E[F |\FFF^X_t] \mathds{1}_A\big] E[\mathds{1}_B] \\
&=E\big[E[F|\FFF^X_{t_i}] \mathds{1}_A \big] E[\mathds{1}_B]= E\big[E[F |\FFF^X_{t_i}] \mathds{1}_C\big].
\end{align*}
Therefore, by the Dynkin class theorem,
\[
E\big[E[F |\FFF^X_t]\big|\FFF^Z_{t_i}\big]=E[F |\FFF^X_{t_i}]
\]
since $\FFF^Z_{t_i}=\sigma(\CCC_{t_i})$. This implies that $E[H_{t_i}|\FFF^X_t]=E\big[E[H_{t_i} |\FFF^Z_{t_i}]\big|\FFF^X_t\big]=E[H_{t_i} |\FFF^X_{t_i}]=\Pi^X(H)_{t_i}$, and therefore
\[
E[H_{t_i} (X_{t_{i+1}} -X_{t_i})|\FFF^X_t]
=E[H_{t_i}|\FFF^X_t] (X_{t_{i+1}} -X_{t_i})
=\Pi^X(H)_{t_i} (X_{t_{i+1}} -X_{t_i}).
\]
On the other hand, we have for the second term that
\begin{align*}
E[H_{t_i} (Y_{t_{i+1}} -Y_{t_i})|\FFF^X_t]
&=E[E[H_{t_i} (Y_{t_{i+1}} -Y_{t_i})|\FFF^Y_{t_i}\vee \FFF^X_t]|\FFF^X_t] \\
&=E[H_{t_i} E[(Y_{t_{i+1}} -Y_{t_i})|\FFF^Y_{t_i}\vee \FFF^X_t]|\FFF^X_t]=0,
\end{align*}
and we can hence conclude that
\[
\left(\int H d Z\right)^{\PPP^X}_t=\lim_{k \to \infty} \sum_{t_i \in \pi^{n_k}} \Pi^X(H)_{t_i} (X_{t_{i+1}} -X_{t_i})
=\int_0^t \Pi^X(H) \, d X,
\]
since $\Pi^X(H)$ remains bounded and left continuous by the general theory of stochastic processes. This proves the claim for $H \in b \LL$. The result is extended first to $H \in b \PPP$ and then to $H \in \HHHH^2(Z)$ by applying respectively the bounded and the monotone convergence theorem. As this procedure is fairly standard, the details are left to the reader.
\end{proof}

In particular, if the predictable projection $H^{\PPP^X}$ exists for $H \in \HHHH^2(Z)$, then $\left(\int H d Z\right)^{\PPP^X}_t=\int_0^t H^{\PPP^X} d X$ $P$-a.s. for all $t \ge 0$. We end this section by observing that \autoref{predproj} immediately gives us an explicit expression for the predictable projection on $\FF^X$ of two important classes of stochastic processes.
\begin{cor} \label{powersonherm} For any $t \ge 0$, the following assertions hold:
\begin{enumerate}
\item $\left(e^{a \cdot Z_t}\right)^{\PPP^X}=\EEE(a \cdot X)_t$ $P$-a.s. for all $a \in \RR^d$ and $t \ge 0$.

\item Let $\alpha=(\alpha_1, \cdots, \alpha_d) \in \NN^d$ denote a multi-index. Then, for $t \ge 0$,
\[
(Z_t^\alpha)^{\PPP^X} =H_\alpha (t,X_t) \; P\text{-a.s.},
\] 
where $z^\alpha:=\prod_{i=1}^d z^{\alpha_i}$ and $H_\alpha$ denotes the $d$-dimensional generalized Hermite polynomial of degree $\alpha$, defined by
\[
H_\alpha (t,X_t):=\prod_{i=1}^d H_{\alpha_i} (t,X^i_t).
\]
In other words, the powers of the conformal Brownian motion project onto the corresponding Hermite polynomials.
\end{enumerate}
\end{cor}

\section{Expansions in Hermite polynomials}

The result of \autoref{powersonherm} (\emph{ii}) is particularly interesting because of the importance of Hermite polynomials in stochastic analysis, in particular in regards to their connection to iterated stochastic integrals and to the Wiener chaos expansion. Thus, expansions in Hermite polynomials and some of their properties will be examined more in detail in this section.

In the following, we denote by $K_n$ the homogeneous Wiener chaos of degree $n$ generated by $(X_t)_{t \in [0,1]}$. First of all, we recall that the Ornstein-Uhlenbeck semigroup $(T_t)_{t \ge 0}$ is defined, for $t \ge 0$ and $F \in L^2(\sigma(X_1))$ by
\[
T_t F:=\sum_{n=0}^\infty e^{- n t} \JJJ_n F,
\]
where $\JJJ_n$ denotes the orthogonal projection on $K_n$. It is well known that the properties of the Ornstein-Uhlenbeck semigroup lead to useful comparison results about the $L^p$-norms on the Wiener chaos. In particular, $(T_t)_{t \ge 0}$ enjoys the following hypercontractivity property:
\begin{prop}\label{hyperorn}
Assume that we have constants $1<p<q<\infty$ and $t>0$ such that
\[
e^t \ge \bigg(\frac{q-1}{p-1}\bigg)^{1/2}.
\]
Then we have that, for all $F \in L^p(\sigma(X_1))$,
\[
\|T_t F\|_q \le \|F\|_p.
\]
\end{prop}

The result can be found, for instance, in Nualart \cite{Nua}. It is then possible to derive the following estimate:

\begin{lem}\label{hyper}
Let $V_n$ be a random variable in $K_n$. Then, for $1 <p < q < \infty$ we have that
\[
\|V_n\|_{q} \le\bigg(\frac{q-1}{p-1}\bigg)^{n/2} \|V_n\|_{p}.
\]
\end{lem}

\begin{proof}
It is well known that, by applying the operator $T_t$ to $V_n$, we get that 
\[
T_t V_n = e^{- n t} V_n.
\]
We now choose $t > 0$ such that $e^t = \big(\frac{q-1}{p-1}\big)^{1/2}$. Then, \autoref{hyperorn} implies that
\[
\bigg(\frac{q-1}{p-1}\bigg)^{-n/2} \|V_n\|_{q} = e^{- n t}\|V_n\|_{q}= \|T_t V_n\|_{q} \le \|V_n\|_{p}.
\qedhere
\]
\end{proof}

Moreover, with the help of the well known interpolation of H\"older's inequality, we can derive from \autoref{hyper} the following inequality.
\begin{lem}\label{hyper2}
Let $V_n$ is a random variable in $K_n$, and $p > 1$. Then:
\[
\|V_n\|_{p} \le e^{np/2} \|V_n\|_{1}.
\]
\end{lem}
\begin{proof}
Let $q > p$ be arbitrary, and let $\theta=\theta(p,q) \in (0,1)$ be such that $\frac{1}{p}=\frac{1-\theta}{q}+ \theta$. Then, the interpolation of H\"older's inequality and \autoref{hyper} yield that
\[ \|V_n\|_p \le \|V_n\|_q^{1- \theta} \|V_n\|_1^\theta \le \bigg(\frac{q-1}{p-1}\bigg)^{n(1- \theta)/2} \|V_n\|_{p}^{1- \theta} \|V_n\|_1^\theta.
\]
By rearranging the terms, this gives us that
\[ \|V_n\|_p \le \bigg(\frac{q-1}{p-1}\bigg)^{\frac{n(1- \theta)}{2\theta}}\|V_n\|_1.
\]
The claim then follows by observing that
\[
\inf_{q \in (p,\infty)} \bigg(\frac{q-1}{p-1} \bigg)^{\frac{n(1- \theta(p,q))}{2\theta(p,q)}}
=\lim_{q \to p^+} \bigg(\frac{q-1}{p-1} \bigg)^{\frac{n(1- \theta(p,q))}{2\theta(p,q)}} = e^{np/2}.
\qedhere
\]
\end{proof}

Even though the constant $e^{np/2}$ could be further optimized, it is sufficiently small for our purpose. From now on, we will write $L^p$ for the space $L^p(\Omega, \FFF,P)$, $p \ge 1$. Because of the well known fact that $(H_\alpha(t,X_t))_{\alpha \in \NN^d}$ forms a complete basis of $L^2(\sigma(X_t))$, we will consider in the sequel series associated to the system $(H_\alpha(t,X_t))_{\alpha \in \NN^d}$. The hypercontractivity allows to find good estimates for $\|H_\alpha(t,X_t)\|_p$: indeed, for $p \ge 2$ we have that
\[
\|H_\alpha(t,X_t)\|_p \le e^{|\alpha|p/2} (\alpha! t^{|\alpha|})^{1/2},
\]
whereas for $1\le p \le 2$ we get that
\[
\|H_\alpha(t,X_t)\|_p \ge \|H_\alpha(t,X_t)\|_1 \ge e^{-|\alpha|/2} (\alpha! t^{|\alpha|})^{1/2}.
\]

As a consequence of the hypercontractivity, we can now derive the second main result of this chapter: this extends, via the corresponding Hermite series, the explicit expression of \autoref{powersonherm} to a class of $L^p$-martingales.
\begin{thm}\label{extension_Lp}
For $p> 1$ and $T>0$, let $(M_t)_{t \in [0,T]}$ be an $L^p$-martingale such that $M_t$ is $\sigma(X_t)$-measurable for all $t \in [0,T]$ (i.e. $M_t$ is of the form $g(t,X_t)$ for some function $g$). Moreover, define $b_\alpha$, for $\alpha \in \NN^d$, by
\[
b_\alpha:=\frac{E[M_t H_\alpha(t,X_t)]}{\|H_\alpha(t,X_t)\|_2^2}=\frac{E[M_t H_\alpha(t,X_t)]}{\alpha! t^{|\alpha|}}.
\]
Then, the function $f:\CC^d \to \CC$, $f(z):=\sum_{\alpha \in \NN^d} b_\alpha z^\alpha$, is well defined, is analytic of order $2$ and can be represented for $z \in \CC^d$ and $t \in [0,T]$ as
\begin{equation} \label{expression_f}
f(z)=E\bigg[M_t \exp \bigg(\frac{1}{t} \Big(z \cdot X_t - \frac{z^T z}{2}\Big)\bigg) \bigg].
\end{equation}

Moreover, let $S< (p \vee p^\ast)^{-1} \, T$, where $p^\ast=\frac{p}{p-1}$ is the conjugate exponent of $p$. Then, $(f(Z_s))_{s \in [0,S]}$ is an $L^p$-martingale such that
\[
\big(f(Z_s)\big)^{\PPP^X}=M_s, \quad s \in [0,S].
\]
\end{thm}

\begin{proof}
First of all, we show that the series $\sum_{\alpha \in \NN^d} b_\alpha z^\alpha$ is absolutely convergent for all $z \in \CC^d$. Since $\prod_{i=1}^d |z_i|^{\alpha_i} \le |z|^{|\alpha|}$, it is sufficient to prove that
\[
\sum_{n=0}^\infty \bigg(\sum_{|\alpha|=n} |b_\alpha|\bigg) |z|^n < \infty.
\]  
By H\"older's inequality and \autoref{hyper}, it is easy to verify that
\begin{align*}
\bigg(\sum_{|\alpha|=n} |b_\alpha|\bigg)^{1/n} &\le\bigg(\sum_{|\alpha|=n}\frac{\|M_t\|_p\|H_\alpha(t,X_t)\|_{p^\ast}}{\alpha!t^{|\alpha|}}\bigg)^{1/n} \\
&\le\bigg(\|M_t\|_p \sum_{|\alpha|=n}\big(({p^\ast}-1)\vee 1 \big)^{|\alpha|/2} \; \frac{\|H_\alpha(t,X_t)\|_2}{\alpha!t^{|\alpha|}}\bigg)^{1/n} \\
&= \bigg(\bigg(\frac{1/(p-1)\vee 1}{t}\bigg)^{n/2} \|M_t\|_p \sum_{|\alpha|=n} \frac{1}{\sqrt{\alpha!}}\bigg)^{1/n}.
\end{align*}
Therefore, by the multinomial theorem,
\begin{align*}
\bigg(\sum_{|\alpha|=n} |b_\alpha|\bigg)^{1/n} &\le \sqrt{\frac{1/(p-1)\vee 1}{t}} \bigg(\|M_t\|_p \sqrt{\big|\big\{\alpha \, | \, |\alpha|=n\big\}\big|}\sqrt{\sum_{|\alpha|=n} \frac{1}{\alpha!}}\bigg)^{1/n} \\
&\le \sqrt{\frac{1/(p-1)\vee 1}{t}} \bigg(\|M_t\|_p \frac{d^{n/2}}{\sqrt{n!}} \sqrt{\sum_{|\alpha|=n} \binom{n}{\alpha}}\bigg)^{1/n} \\
&= d \sqrt{\frac{1/(p-1)\vee 1}{t}} \bigg(\|M_t\|_p \frac{1}{\sqrt{n!}}\bigg)^{1/n}.
\end{align*}
The last term converges for $n \to \infty$ to $0$ because of Stirling's approximation. This shows that $f$ is well defined and analytic. We now prove the representation \eqref{expression_f}. By applying \autoref{powersonherm} and the dominated convergence theorem, it is easy to check that, for all $t > 0$,
\begin{align*}
f(z)=\sum_{\alpha \in \NN^d} \frac{1}{\alpha! t^{|\alpha|}} E[M_t H_\alpha(t,X_t)] \cdot z^\alpha 
&=\sum_{\alpha \in \NN^d} \frac{1}{\alpha! t^{|\alpha|}} E[M_t Z_t^\alpha] \cdot z^\alpha \\ 
&= E\left[M_t \sum_{\alpha \in \NN^d} \frac{z^\alpha Z_t^\alpha}{\alpha! t^{|\alpha|}} \right].
\end{align*}
Hence, by the multinomial theorem, 
\begin{align*}
f(z)= E\bigg[M_t \sum_{n \in \NN} \sum_{|\alpha|=n} \frac{z^\alpha Z_t^\alpha}{\alpha! t^{|\alpha|}} \bigg]
&= E\bigg[M_t \sum_{n \in \NN} \frac{1}{n! t^n}\sum_{|\alpha|=n} \binom{n}{\alpha} z^\alpha Z_t^\alpha \bigg]  \\
&= E\bigg[M_t \sum_{n \in \NN} \frac{(z \cdot Z_t)^n}{n! t^n} \bigg],
\end{align*}
and the desired representation follows by computing that
\begin{align*}
f(z)= E\bigg[M_t \exp\bigg(\frac{z \cdot Z_t}{t} \bigg) \bigg] 
&= E\bigg[M_t \exp\bigg(\frac{z \cdot X_t}{t} \bigg) E\bigg[\exp\bigg( i \, \frac{z \cdot Y_t}{t} \bigg)\bigg] \bigg]\\ 
&= E\bigg[M_t \exp \bigg(\frac{1}{t} \Big(z \cdot X_t - \frac{z^T z}{2}\Big)\bigg) \bigg].
\end{align*}
On the other hand, by applying H\"older's inequality to this representation we can verify that
\begin{align}
|f(z)| &\le \|M_t\|_p \left\|\exp \left(\frac{1}{t} \left( z \cdot X_t - \frac{z^T z}{2}\right)\right)\right\|_{p^\ast} \nonumber\\
&=\|M_t\|_p E\left[\left|\exp\left(\frac{{p^\ast}}{t} z \cdot X_t\right)\right|\right]^{1/{p^\ast}} \exp\bigg(- \frac{\Real(z^T z)}{2t}\bigg) \nonumber\\
&= K_t \exp\bigg(\frac{1}{2t}(({p^\ast}-1) \Real(z)^2 + \Imag(z)^2)\bigg) \nonumber\\
&\le K_t \exp\bigg(\frac{({p^\ast}-1) \vee 1}{2t} |z|^2\bigg), \label{estf}
\end{align}
where $K_t:=\|M_t\|_p$, and hence $f$ is analytic of order $2$. It remains to consider the process $(f(Z_s))_{s \in [0,S]}$. Because of the choice of $S$, we get that $f(Z_s) \in L^p$ for all $s \in [0,S]$; namely, by taking some $t \in [0,T]$ such that $t > (p \vee p^\ast) s$, the estimate \eqref{estf} implies that
\[
E\Big[|f(Z_s)|^p\Big] \le K_t E\bigg[\exp\bigg(\frac{p \vee p^\ast}{2t} |Z_s|^2\bigg)\bigg] < \infty.
\]
Then, the fact that $(f(Z_s))_{s \in [0,S]}$ is a martingale such that $\big(f(Z_s)\big)^{\PPP^X}=M_s$ for $s \in [0,S]$ is obtained by applying Fubini's theorem and \autoref{powersonherm} to the integral representation \eqref{expression_f} for $f$: the details are left to the reader.
\end{proof}

\noindent As a consequence, we obtain the following result for martingales on $[0,\infty)$:
\begin{cor}\label{cor_extension_Lp}
Let $p> 1$, and assume that $(M_t)_{t \in [0,\infty)}$ is an $L^p$-martingale such that $M_t$ is $\sigma(X_t)$-measurable for all $t \in [0,T]$, and let $f$ as in \autoref{extension_Lp}. Then, $(f(Z_t))_{t \in [0,\infty)}$ is an $L^p$-martingale such that
\[
\big(f(Z_t)\big)^{\PPP^X}=M_t, \quad t \in [0,\infty).
\]
\end{cor}

By analyzing the proof of \autoref{extension_Lp}, we immediately notice that the result cannot hold for $p=1$, due to the unboundedness of the Hermite polynomials. It is however possible to relate the convergence in $L^1$ of an Hermite series to that in $L^p$, $p>1$. 

\begin{prop}\label{hermseries}
Let $t \ge 0$, and assume that the family $( b_\alpha H_\alpha(t,X_t) )_{\alpha \in \NN^d}$ is bounded in $L^1$. Then, for $p > 1$, $\sum_{\alpha \in \NN^d} b_\alpha H_\alpha(s,X_s)$ converges absolutely in $L^p$ for $s < \frac{t}{d^2 e^p}$.
\end{prop}

\begin{proof}
Let $s \ge 0$ be arbitrary. As a consequence of the estimates on Hermite polynomials and of the scaling property of Brownian motion we obtain that
\begin{align}
\sum_{\alpha \in \NN^d} |b_\alpha| \|H_\alpha(s,X_s)\|_p
&\le \sum_{\alpha \in \NN^d} |b_\alpha| e^{|\alpha|p/2} \|H_\alpha(s,X_s)\|_1 \nonumber \\
&= \sum_{\alpha \in \NN^d} |b_\alpha| \bigg(e^{p} \frac{s}{t}\bigg)^{|\alpha|/2} \|H_\alpha(t,X_t)\|_1. \label{Lp_conv}
\end{align}
Let $K$ denote the $L^1$-bound on the family $( b_\alpha H_\alpha(t,X_t) )_{\alpha \in \NN^d}$, and assume that $s < \frac{t}{d^2 e^p}$. As a consequence of \eqref{Lp_conv}, we then obtain that
\begin{align*}
\sum_{\alpha \in \NN^d} |b_\alpha| \|H_\alpha(s,X_s)\|_p
&\le K \sum_{\alpha \in \NN^d} \bigg(e^{p} \frac{s}{t}\bigg)^{|\alpha|/2} \\
&\le K \sum_{n=0}^\infty \bigg(e^{p} \frac{s}{t}\bigg)^{n/2} \big|\{ \alpha \in \NN^d | \ |\alpha|=n\}\big| \\ 
&\le K \sum_{n=0}^\infty \bigg(e^{p/2} d \sqrt{\frac{s}{t}}\bigg)^n < \infty,
\end{align*}
because of the choice of $s$. This proves the desired absolute convergence.
\end{proof}

As a consequence of \autoref{hermseries}, we also obtain:
\begin{cor}\label{cor_hermseries}
Assume that, for all $t \ge 0$, the series $\sum_{\alpha \in \NN^d} b_\alpha H_\alpha(t,X_t)$ converges, for some rearrangement of the terms, in $L^1$. Then, $\sum_{\alpha \in \NN^d} b_\alpha H_\alpha(t,X_t)$ converges absolutely in $L^p$ for all $t \ge 0$ and $p \ge 1$.
\end{cor}

\section{Widder's representation for Brownian martingales}
As an application of the above properties, we derive a characterization of Widder's theorem about the representation of positive martingales. First of all, we recall how the classical formulation of Widder translates in a probabilistic setting. A purely probabilistic proof of this theorem can be found in \cite{ManYorBook}, but we prefer to include a different presentation for the convenience of the reader. Moreover, we observe that our result hold for any dimension $d \in \NN$.

\begin{thm}\label{widder_original}
Let $X$ be a $d$-dimensional Brownian motion, and suppose that $M_t=g(t,X_t)$ is a continuous martingale such that $g(0,0)=1$ and $g(t,X_t) \ge 0$. Then, there exists a probability measure $\mu$ on $\RR^d$ such that, for all $t \ge 0$,
\[
M_t=\int_{\RR^d} \EEE(v \cdot X)_t \ \mu (d v)\quad P\text{-a.s.}.
\]
\end{thm}

\begin{proof}

Let $f: \RR^d \to \CC$ be a bounded, continuous map. Then, for $s \ge t$ one has that
\[E\left[M_t f\left(\frac{X_t}{t}\right)\right]=E\left[M_s f\left(\frac{X_t}{t} \right)\right] =E\left[M_s f\left(\frac{X_s}{s}+\left(\frac{X_t}{t} - \frac{X_s}{s}\right)\right)\right].\]

Note that the random variables $\left(\frac{X_t}{t} - \frac{X_s}{s}\right)$ and $\frac{X_s}{s}$ are independent as they are uncorrelated in a Gaussian space. Therefore, by conditioning on $\frac{X_s}{s}$ one gets that
\[E\left[M_t f\left(\frac{X_t}{t}\right)\right]=E\left[M_s \int_{\RR^d}f\left(\frac{X_s}{s}+ \sqrt{\frac{1}{t}-\frac{1}{s}} x\right) \frac{\exp (-|x|^2/2)}{(2 \pi)^{d/2}} d x\right].\]

We can now proceed with the construction of the measure $\mu$. By taking $f(x):=\exp( i u \cdot x)$, the previous equality yields that
\[E\left[M_t \exp\left( i u \cdot \frac{X_t}{t}\right)\right]=E\left[M_s \exp\left(i u \cdot\frac{X_s}{s}\right)\right] \exp\left(-\frac{1}{2}\left(\frac{1}{t}-\frac{1}{s}\right) |u|^2\right).\]

On the other hand, the process $(M_t)_{t \ge 0}$ is by assumption a density process, and is therefore associated to a measure $Q$ on $C(\RR^+,\RR^d)$ with $Q\big|_{\FFF_t} << P\big|_{\FFF_t}$ via the Radon-Nikodym theorem. This gives us that
\[E_Q\left[\exp\left( i u \cdot \frac{X_t}{t}\right)\right]=E_Q\left[\exp\left(i u \cdot \frac{X_s}{s}\right)\right] \exp\left(-\frac{1}{2}\left(\frac{1}{t}-\frac{1}{s}\right) |u|^2\right).\]
By taking $t=1$, this implies in particular that
\[\varphi_{X_1}^Q(u)=\varphi_{\frac{X_s}{s}}^Q(u)\exp\left(-\frac{1}{2}\left(1-\frac{1}{s}\right) |u|^2\right),\]
where $\varphi_{Y}^Q$ denotes the characteristic function of the random variable $Y$ under the measure $Q$.

As a consequence, we get that $\varphi_{\frac{X_s}{s}}^Q(u)$ converges pointwise for $s \to \infty$ to a continuous function $\varphi_{X_1}^Q(u) \exp\left(\frac{1}{2} |u|^2\right)$. Therefore, L\'evy's continuity theorem yields the existence of a measure $\mu$ such that $\frac{X_s}{s}$ converges weakly to $\mu$ under $Q$ as $s \to \infty$.

We now have to check that $\mu$ satisfies the desired property. By the above convergence in distribution, we have that, for any $f$ bounded and continuous and for any fixed $t> 0$,
{\small
\[
E_Q\!\left[\int_{\RR^d} \! f\!\left(\frac{X_s}{s}+\sqrt{\frac{1}{t}-\frac{1}{s}} x\right) \frac{e^{-|x|^2/2}}{(2 \pi)^{d/2}} d x\right] \to \int_{\RR^d}\!\int_{\RR^d}\! f\!\left(\!v+\sqrt{\frac{1}{t}} x\right) \frac{e^{-|x|^2/2}}{(2 \pi)^{d/2}} d x \ \mu (dv)
\]
}
when $s \to \infty$. Since the left hand side equals $E\left[M_t f\left(\frac{X_t}{t}\right)\right]$ for any $s \ge t$, we get that
\[E\left[M_t f\left(\frac{X_t}{t}\right)\right]=\int_{\RR^d}\int_{\RR^d} f\left(v+\sqrt{\frac{1}{t}} x\right) \frac{e^{-|x|^2/2}}{(2 \pi)^{d/2}} d x \ \mu (dv).\]

On the other hand, by using Girsanov's theorem one can easily check that, for any $f$ bounded and continuous,
\begin{align*}
E\left[\int_{\RR^d} \EEE(v X)_t \ \mu (d v) f\left(\frac{X_t}{t}\right)\right] 
&= \int_{\RR^d} E\left[\EEE(v X)_t  \ f\left(\frac{X_t}{t}\right)\right] \mu (d v) \\ 
&= \int_{\RR^d} E\left[f\left(\frac{X_t}{t}+ v\right)\right] \mu(d v) \\ 
&= \int_{\RR^d} \int_{\RR^d}f\left(v+\sqrt{\frac{1}{t}} x\right) \frac{e^{-|x|^2/2}}{(2 \pi)^{d/2}} d x \ \mu (dv).
\end{align*}

By approximation arguments, this equality can then be extended to any $f$ bounded and measurable, obtaining the desired representation for $M_t$.
\end{proof}

The condition that $g(t,X_t)$ forms a martingale on the whole interval $[0,\infty)$ is essential. This can easily be shown by separation arguments: let $d=1$, $T >0$, and denote by ${\mathcal K}$ the set of all the martingales having the desired representation, i.e.
\begin{align*}
{\mathcal K}=\big\{N=(N_t)_{0 \le t \le T} \ \big| &\text{ There is a probability measure } \mu \text{ such that } \\
& N_t=\int_\RR \EEE(v X)_t \ \mu (d v), 0 \le t \le T \big\}.
\end{align*}
Clearly, ${\mathcal K}$ is a convex set. Now, for an arbitrary $f \in L^\infty(\RR^d)$ and any $N \in {\mathcal K}$ we have that
\begin{align*}
E\left[f(X_t)N_t\right]&=\int_{\Omega}\int_\RR f(X_t)\EEE(v X)_t \ \mu (d v) d P \\
&= \int_\RR E\left[f(X_t)\EEE(v X)_t\right] \mu (d v) \\
&= \int_\RR E\left[f(X_t+ v t)\right] \mu (d v).
\end{align*}
The choice $f_k(x):= \sin(k x)$ gives
\begin{align*}
E\left[f_k(X_t)N_t\right]&=\int_\RR E\left[\sin(k X_t+ k v t)\right] \mu (d v) \\ &= \int_\RR E\left[\cos(k X_t)\right]\sin(k v t) \mu (d v) \\ 
&\ge E\left[\cos(k X_t)\right] \inf_{v \in \RR} \sin(k v t) = - e^{-k^2 t/2}.
\end{align*}
However, by defining 
\[
M_T=\frac{\mathds{1}_{\{\sin(k X_T)<- e^{-k^2 T/2}\}}}{P(\sin(k X_T)<- e^{-k^2 T/2})}
\]
and by setting $M_t:=E[M_T|\FFF_t]=g(t,X_t)$, we get a martingale on $\left[0,T\right]$ of the desired form and such that $E\left[f_k(X_t)M_t\right] < - e^{-k^2 t/2}$. This proves the claim.

On the other hand, \autoref{widder_original} can easily be extended to any continuous $L^1$-bounded Brownian martingale on $[0,\infty)$. First of all, we recall the well known Krickeberg decomposition for $L^1$-bounded martingales:
\begin{thm}
Let $M$ denote a martingale on $[0, \infty)$. Then $M$ is $L^1$-bounded if and only if it can be written ($P$-a.s.) as the difference of two positive martingales $M^1$, $M^2$. Moreover, one can choose $M^1$, $M^2$ so that
\[
\sup_{t \ge 0} \|M_t\|_1=E[M^1_0]+E[M^2_0],
\] 
and the decomposition is then given by
\[
M_t^1=\sup_{s \ge t} E[M_s^+|\FFF_t], \quad M_t^2=\sup_{s \ge t} E[M_s^-|\FFF_t] \quad P\text{-a.s.}, \quad t \ge 0.
\]
\end{thm}

The proof of this well known result can be found for instance in \cite{DelMeyBook}. Krickeberg's decomposition allows us to extend Widder's representation theorem, obtaining the following result:

\begin{prop}\label{Krickeberg}
Let $(M_t)_{t \ge 0}$ be a continuous, $L^1$-bounded martingale of the form $M_t=g(t,X_t)$. Then there is a signed measure $\mu$ on $\RR^d$ such that
\[
M_t=\int_{\RR^d} \EEE(v \cdot X)_t \ \mu (d v) \quad P\text{-a.s.} \text{ for all } t \ge 0.
\]
Moreover, we have that
\[
\sup_{t \ge 0} \|M_t\|_1 = \|\mu\|.
\]
\end{prop}

\begin{proof}

We apply the Krickeberg decomposition to $M$, obtaining two positive martingales $M^1$, $M^2$ such that $M=M^1-M^2$ and $\sup_{t \ge 0} \|M_t\|_1=M^1_0+M^2_0$.

By construction, $M^i$ is of the form $M^i=f^i(t,X_t)$: indeed,
\[
M_t^1=\sup_{s \ge t} E[M_s^+|\FFF_t]=\sup_{s \ge t} E[M_s^+|X_t],
\]
since $M_s^+$ is $\sigma(X_s)$-measurable. We can assume, without loss of generality, that $M^i_0 \neq 0$. Then, we can apply Widder's representation to the positive martingales $\frac{M^1}{M^1_0}$, $\frac{M^2}{M^2_0}$, obtaining two probability measures $\hat{\mu}^1$, $\hat{\mu}^2$ such that
\[
\frac{M^i_t}{M^i_0}=\int_{\RR^d} \EEE(v \cdot X)_t \ \hat{\mu}^i (d v)\quad P\text{-a.s.}.
\]
We thus set $\mu^i:= M^i_0 \hat{\mu}^i$ and $\mu:=\mu^1- \mu^2$. Then, for all $t$,
\[
M_t=M_t^1 - M_t^2=\int_{\RR^d} \EEE(v \cdot X)_t \ \mu (d v)\quad P\text{-a.s.}.
\]
On the other hand, we have that
\[
\|M_t\|_1 \le \int_\Omega \int_{\RR^d} \EEE(v \cdot X)_t \ |\mu| (d v) d P = \int_{\RR^d} E\left[\EEE(v \cdot X)_t\right] \ |\mu| (d v) = \|\mu\|, 
\]
so that we finally get that
\[
\|\mu^1\|+\|\mu^2\|=M_0^1+M_0^2=\sup_{t \ge 0} \|M_t\|_1 \le \|\mu\| \le \|\mu^1\|+\|\mu^2\|.
\qedhere
\]
\end{proof}

\section{A characterization of Widder's measure}

We now present the aforementioned characterization of the measure $\mu$ appearing in \autoref{Krickeberg}. To the best of our knowledge, this characterization is new and shows interesting analogies with results from Fourier analysis.

\begin{thm}\label{widder_characterization}
Let $(g(t,X_t))_{t \ge 0}$ be a continuous $L^1$-bounded martingale on $[0,\infty)$ with Widder's representation $g(t,X_t)=\int_{\RR^d} \EEE(v \cdot X)_t \ \mu (d v)$. Then, the following assertions hold:
\begin{enumerate} 
\item If the measure $|\mu|$ has quadratic exponential moments of all orders, i.e. 
\[
\int_{\RR^d} e^{\lambda |v|^2}|\mu| (d v) < \infty \text{ for all } \lambda > 0,
\]
then there is a family $(b_\alpha)_{\alpha \in \NN^d}$ of coefficients in $\RR$ such that, for all $t \ge 0$, the series $\sum_{\alpha \in \NN^d} b_\alpha H_\alpha(t,X_t)$ converges absolutely to $g(t,X_t)$ in $L^1$, and therefore in $L^p$ for all $p > 1$. The coefficients $b_\alpha$ can be represented as
\[
b_\alpha=\frac{1}{\alpha!}\int_{\RR^d} v^\alpha \ \mu (d v), \quad \alpha \in \NN^d.
\]
 \item Conversely, if $\mu$ is positive and there is a family $(b_\alpha)_{\alpha \in \NN^d}$ such that the series $\sum_{\alpha \in \NN^d} b_\alpha H_\alpha(t,X_t)$ converges, for some rearrangement of the terms, to $g(t,X_t)$ in $L^1$ for all $t \ge 0$, then $\mu$ has quadratic exponential moments of all orders. Moreover, the function $f(z):=\int_{\RR^d} e^{v \cdot z}\mu (d v)$, $z \in \CC^d$, is well defined, analytic of order $2$, and can be represented, for $z \in \CC^d$ and $t > 0$, as
\[
f(z)=\sum_{\alpha \in \NN^d} b_\alpha z^\alpha=E\bigg[g(t,X_t) \exp \bigg(\frac{1}{t} \Big(z \cdot X_t - \frac{z^T z}{2}\Big)\bigg) \bigg].
\]
\end{enumerate}
\end{thm}

\begin{proof}
We can assume, without loss of generality, that $|\mu|$ has total mass $1$. We first show (\textit{i}): clearly, we have that
\[
|g(t,X_t)|^2=\bigg|\int_{\RR^d} \EEE(v \cdot X)_t \ \mu (d v)\bigg|^2 \le \int_{\RR^d} \EEE(v \cdot X)_t^2 \ |\mu| (d v),
\]
so that
\[
E\big[|g(t,X_t)|^2\big] \le \int_{\RR^d} E\big[\EEE(v \cdot X)_t^2\big]\ |\mu| (d v) =
\int_{\RR^d} e^{t |v|^2}\ |\mu| (d v) < \infty
\]
for all $t \ge 0$. Hence, $g(t,X_t)$ is the limit in $L^2$ of the series $\sum_{\alpha \in \NN^d} b_\alpha H_{\alpha}(t,X_t)$, where $b_\alpha=\frac{1}{\alpha! t^{|\alpha|}} E[g(t,X_t) H_\alpha(t,X_t)]$. A simple application of Fubini's theorem then gives that
\begin{align*}
b_\alpha&=\frac{E\big[\big(\int_{\RR^d} \EEE(v \cdot X)_t \ \mu (d v) \big) H_\alpha(t,X_t)\big]}{\alpha! t^{|\alpha|}} \\
&=\frac{1}{\alpha! t^{|\alpha|}} \int_{\RR^d}  E[ \EEE(v \cdot X)_t H_\alpha(t,X_t)] \ \mu (d v) =\frac{1}{\alpha!} \int_{\RR^d} v^\alpha \ \mu (d v). 
\end{align*}

We now prove the second implication. Since the series $\sum_{\alpha \in \NN^d} b_\alpha H_\alpha(t,X_t)$ converges for some rearrangement to $g(t,X_t)$ in $L^1$, \autoref{cor_hermseries} implies that the same series converges absolutely in $L^2$ to $g(t,X_t)$ for all $t \ge 0$. Hence, we get that
\[
\infty > E\big[g(t,X_t)^2\big]=E\bigg[\bigg(\int_{\RR^d} \EEE(v \cdot X)_t \mu (d v)\bigg)^2\bigg]=\int_{\RR^d} \int_{\RR^d} e^{(v \cdot w) t} \mu(d v) \mu(d w)
\]
for all $t \ge 0$. This is enough to show the existence of quadratic exponential moments of $\mu$: indeed, by setting $C_t:=\|g(t,X_t)\|_2$ we obtain that, for all $K >0$,
\begin{equation}\label{mu_v_w}
(\mu \otimes \mu)\big(v \cdot w > K\big) \le C_t \exp(- t K).
\end{equation}
On the other hand it is easy to verify that, for $i \in \{1,\cdots, d\}$ and $(\eps_1,\cdots, \eps_d) \in \{-1,1\}^d$, the set
\begin{align*}
\AAA_{K,i,(\eps_1,\cdots, \eps_d)}:=&\{v\in \RR^d | \, |v_i|> \sqrt{K}, \;  \mathrm{sign}(v_j)=\eps_j \; \forall \ j \} \\
&\times \{w\in \RR^d | \, |w_i|> \sqrt{K}, \; \mathrm{sign}(w_j)=\eps_j \; \forall \ j \}
\end{align*}
is a subset of $\{(v,w) \in \RR^d \times \RR^d \, | \, v\cdot w > K\}$. Thus, \eqref{mu_v_w} implies that, for $i \in \{1,\cdots, d\}$ and $(\eps_1,\cdots, \eps_d) \in \{-1,1\}^d$,
\begin{equation}\label{mu_v}
\mu\big(\{v\in \RR^d | \, |v_i|> \sqrt{K}, \;  \mathrm{sign}(v_j)=\eps_j \; \forall \ j \}\big) \le C_t^{1/2} \exp\bigg(- \frac{t}{2} \, K\bigg).
\end{equation}
Now, if $v \in \RR^d$ is such that $|v|^2 > K$, then there is an $i \in \{1,\cdots, d\}$ such that $|v_i|>\sqrt{K/d} \, $: hence, \eqref{mu_v} gives that
\[
\mu\big(|v|^2 > K\big) \le d \, 2^d \, C_t^{1/2} \exp\bigg(- \frac{t}{2 d} \, K\bigg).
\]
Since $t \ge 0$ is arbitrary, this finally implies for all $\lambda \ge 0$ that
\[
\int_{\RR^d} e^{\lambda |v|^2} \mu (d v) < \infty. \qedhere
\]
\end{proof}

We conclude this article with an example illustrating the fact that the quadratic exponential moments of $\mu$ are needed in order to have an expansion in Hermite series of the corresponding martingale. Let $\mu$ denote the standard Gaussian measure on $\RR$: we can then verify that, for $t \ge 0$,
\[
g(t,X_t)=\int_\RR \EEE(v \cdot X)_t \ \mu (d v) = \frac{1}{\sqrt{t+1}} \exp \bigg(\frac{X_t^2}{2(t+1)}\bigg).
\]
Since not all the quadratic exponential moments of $\mu$ exist, \autoref{widder_characterization} implies that $g(t,X_t)$ cannot be represented as an Hermite series in $L^1$. 
Moreover, it is easy to check that $g(t,X_t)$ corresponds to the counterexample introduced in a deterministic setting by Pollard \cite{Pol} in order to prove that the Hermite polynomials do not form a basis of $L^p$ for $p \neq 2$, and his conclusions can then be recovered from the results on $L^p$-convergence proved in Section $3$.

We can therefore observe that \autoref{widder_characterization} gives a full explanation of Pollard's counterexample: namely, the non-convergence of the corresponding Hermite series is simply due to the non-existence of quadratic exponential moments of any order for the corresponding Widder's measure. This observation allows us to construct several other counterexamples whose Hermite series do not converge in $L^1$: after choosing a measure on $\RR$ which does not have quadratic exponential moments of all orders, it suffices to consider the martingale given by the corresponding Widder representation. 

Recently, there was a renewed interest in Widder's representation in connection with boundary crossing problems for Brownian motion, see Alili and Patie \cite{AliPat}.


\end{document}